\documentclass[a4paper,11pt]{amsart}

\usepackage{mathpazo}
\usepackage{amssymb}
\usepackage[pagebackref,
    ,pdfborder={0 0 0}
    ,urlcolor=black,a4paper,hypertexnames=false]{hyperref}
\hypersetup{pdfauthor={Daniel Fauser, Clara L\"oh},pdftitle=Exotic finite functorial semi-norms on singular homology}

\usepackage{pgf}
\usepackage{tikz}
\usetikzlibrary{decorations.pathmorphing}
\usepackage[all]{xy}
\SelectTips{cm}{}

\newtheorem{thm}{Theorem}[section]
\newtheorem{prop}[thm]{Proposition}

\newtheorem{cor}[thm]{Corollary}

\newtheorem{question}[thm]{Question}

\theoremstyle{remark}
\newtheorem{rem}[thm]{Remark}
\newtheorem{exa}[thm]{Example}

\theoremstyle{definition}
\newtheorem{defi}[thm]{Definition}

\usepackage{amsfonts}
\newcommand{\Z}{\mathbb{Z}}
\newcommand{\Q}{\mathbb{Q}}
\newcommand{\R}{\mathbb{R}}
\newcommand{\N}{\mathbb{N}}

\DeclareMathOperator{\Hom}{Hom}

\DeclareMathOperator{\map}{map}

\def\epsilon{\varepsilon}

\DeclareMathOperator{\Top}{{\sf Top}}

\DeclareMathOperator{\Vect}{{\sf Vect}}
\DeclareMathOperator{\sn}{{\sf sn}}

\def\args{\;\cdot\;}

\def\sv#1{\|#1\|}

\usepackage{color}
\usepackage{pdfcolmk}

\def\draftinfo{}

\author{Daniel Fauser}
\author{Clara L\"oh}
\title[Exotic finite functorial semi-norms]%
      {Exotic finite functorial semi-norms\\ on singular homology}
\date{\today.\ \copyright{\ D.~Fauser, C.~L\"oh 2016}. 
    This work was supported by the CRC~1085 \emph{Higher Invariants} 
    (Universit\"at Regensburg).
    \draftinfo\\
     MSC~2010 classification: 55N10, 57N65}
\begin{document}

\begin{abstract}
  Functorial semi-norms on singular homology give refined ``size''
  information on singular homology classes. A fundamental example is
  the $\ell^1$-semi-norm. We show that there exist finite functorial
  semi-norms on singular homology that are exotic in the sense that
  they are \emph{not} carried by the $\ell^1$-semi-norm.
\end{abstract}

\phantom{.}
\vspace{-.1\baselineskip}

\maketitle

\section{Introduction}

Functorial semi-norms on singular homology give refined ``size''
information on singular homology classes. On the one hand, functorial
semi-norms lead to obstructions for mapping degrees. On the other
hand, mapping degrees allow to construct functorial semi-norms on
singular homology. A fundamental example of a finite functorial
semi-norm on singular homology is the $\ell^1$-semi-norm underlying
the definition of simplicial volume (see Section~\ref{sec:fsn} for the
definitions).

While the general classification of functorial semi-norms on singular
homology is out of reach, one can ask for the role of the
$\ell^1$-semi-norm among all finite functorial
semi-norms~\cite[Question~5.8]{crowleyloeh}. A simple rescaling
manipulation shows that not all finite functorial semi-norms on singular
homology are dominated by a multiple of the
$\ell^1$-semi-norm~\cite[Section~5]{crowleyloeh}. Relaxing the
domination condition, we introduce the following relation between
functorial semi-norms:

\begin{defi}[carriers of functorial semi-norms]\label{def:carrier}
  Let $d\in \N$.  A functorial semi-norm~$|\cdot|$ on $H_d(\args;\R)$
  \emph{carries} a functorial semi-norm~$|\cdot|'$ if for all topological spaces~$X$ 
  and all~$\alpha \in H_d(X;\R)$ we have
  \[ |\alpha| = 0 \Longrightarrow |\alpha|' = 0. 
  \]
\end{defi}

In these terms, the current paper is concerned with the question
whether every finite functorial semi-norm on singular homology is
carried by the $\ell^1$-semi-norm. 
All finite functorial semi-norms on~$H_d(\args;\R)$ that are
multiplicative under finite coverings are carried by the
$\ell^1$-semi-norm in a strong
sense~\cite[Proposition~7.11]{crowleyloeh}. However, if the
multiplicativity condition is dropped, then exotic finite functorial
semi-norms appear:

\begin{thm}\label{thm:exotic}
  Let $d \in \{3\} \cup \N_{\geq 5}$. Then there exists a finite
  functorial semi-norm on~$H_d(\args;\R)$ that is \emph{not} carried
  by the $\ell^1$-semi-norm.
\end{thm}

In particular, this answers a question by Crowley and L\"oh on
``maximality'' of the $\ell^1$-semi-norm in the
negative~\cite[Question~5.8]{crowleyloeh}.

Our construction of exotic finite functorial semi-norms is based on
the parallel observation about mapping degrees of manifolds:

\begin{thm}\label{thm:stronglyinflexible}
  Let $d \in \{3\} \cup \N_{\geq 5}$. Then there exists a strongly
  inflexible oriented closed connected $d$-manifold~$M$ with~$\sv M = 0$. 
  Moreover, we can choose $M$ to be aspherical.
\end{thm}

In contrast, Theorem~\ref{thm:stronglyinflexible} is clearly wrong in
dimension~$2$.  However, it is an open problem to decide whether all
finite functorial semi-norms on~$H_2(\args;\R)$ are carried by the
$\ell^1$-semi-norm. In the case of dimension~$4$, both
Theorem~\ref{thm:exotic} and~\ref{thm:stronglyinflexible} remain open.

\subsection*{Organisation of this article}

We recall the basic terminology for functorial semi-norms in
Section~\ref{sec:fsn}. Strongly inflexible manifolds are discussed in
Section~\ref{sec:stronglyinflexible}.  The proof of
Theorem~\ref{thm:stronglyinflexible} is given in
Section~\ref{subsec:sinflexsimvol} and Theorem~\ref{thm:exotic} is
then derived in Section~\ref{sec:exotic}. Moreover, we briefly we
explain the relation with secondary simplicial volume in
Section~\ref{subsec:secondarysimvol}.

\section{Functorial semi-norms}\label{sec:fsn}

We begin by recalling the terminology for functorial semi-norms 
and the $\ell^1$-semi-norm in particular.

\subsection{Terminology}\label{subsec:terminology}

In the following, semi-norms are allowed to have values in~$\R_{\geq
  0} \cup \{\infty\}$, where we use the usual conventions that~$a +
  \infty =\infty$ and $b \cdot \infty = \infty$ holds for all~$a \in
  \R_{\geq 0}$ and all~$b \in \R_{>0}$.

\begin{defi}[functorial semi-norm]
  Let $d \in \N$. A \emph{functorial semi-norm} on~$H_d(\args;\R)$ is
  a lift of the functor~$H_d(\args;\R) \colon \Top \longrightarrow
  \Vect_\R$ to a functor~$\Top \longrightarrow \Vect_\R^{\sn}$, where
  $\Vect_\R^{\sn}$ denotes the category of semi-normed $\R$-vector
  spaces with norm non-increasing $\R$-linear maps. More concretely, 
  a functorial semi-norm on~$H_d(\args;\R)$ consists of a choice of 
  a semi-norm~$|\cdot|$ on~$H_d(X;\R)$ for every topological space such that 
  the following compatibility holds: If $f \colon X \longrightarrow Y$ 
  is a continuous map, then 
  \[ \forall_{\alpha \in H_d(X;\R)}\  \bigl| H_d(f;\R)(\alpha)\bigr| \leq |\alpha|. 
  \]
  A functorial semi-norm on~$H_d(\args;\R)$ is \emph{finite} if $|\alpha| < \infty$ 
  for all singular homology classes~$\alpha$ in degree~$d$.
\end{defi}

\begin{rem}
  If $|\cdot|$ is a functorial semi-norm on~$H_d(\args;\R)$ and if
  $f\colon N \longrightarrow M$ is a continuous map between oriented
  closed connected $d$-manifolds, then 
  \[ |\deg f| \cdot \bigl|[M]_\R \bigr| \leq \bigl|[N]_\R \bigr|. 
  \]
  In particular: If $0 < |[M]_\R| < \infty$, then $M$ is strongly
  inflexible (Definition~\ref{def:sinflex}).
\end{rem}

The classical example of a finite functorial semi-norm
on~$H_d(\args;\R)$ is the $\ell^1$-semi-norm (see
Section~\ref{subsec:l1} below).  Other examples of functorial
semi-norms can be constructed by means of manifold topology, e.g., the
products-of-surfaces semi-norm~\cite[Sections~2, 7]{crowleyloeh} or
infinite functorial semi-norms that exhibit exotic behaviour on
certain classes of simply connected spaces of high
dimension~\cite[Theorem~1.2]{crowleyloeh}; such a construction
principle via manifold topology will be recalled in Section~\ref{subsec:genfsn}.

\subsection{The $\ell^1$-semi-norm and simplicial volume}\label{subsec:l1}

For the sake of completeness we include the definition of the
$\ell^1$-semi-norm on singular homology and simplicial volume.

\begin{defi}[$\ell^1$-semi-norm]
  Let $d \in \N$ and let $X$ be a topological space. For a singular chain~$c = \sum_{j=1}^m a_j \cdot \sigma_j 
  \in C_d(X;\R)$ (in reduced form) we define
  \[ |c|_1 := \sum_{j=1}^m |a_j|. 
  \]
  I.e., $|\cdot|_1$ is the $\ell^1$-norm on~$C_d(X;\R)$ associated
  with the basis given by all singular $d$-simplices in~$X$. The
  semi-norm~$\|\cdot\|_1$ on~$H_d(X;\R)$ induced by the norm~$|\cdot|_1$ is the
  \emph{$\ell^1$-semi-norm on~$H_d(X;\R)$}.
\end{defi}

A straightforward calculation shows that the $\ell^1$-semi-norm 
indeed is a functorial semi-norm on~$H_d(\args;\R)$. 
Applying this semi-norm to fundamental classes of
manifolds gives rise to Gromov's simplicial volume~\cite{vbc}:

\begin{defi}[simplicial volume/Gromov norm]
  The \emph{simplicial volume} of an oriented closed connected
  manifold~$M$ is defined by
  \[ \| M \| := \bigl\| [M]_\R \bigr\|_1 \in \R_{\geq 0}. 
  \]
\end{defi}

Geometrically speaking, simplicial volume measures how many singular
simplices are needed to reconstruct the given manifold.  Simplicial
volume allows for interesting applications, linking topology and
geometry of manifolds~\cite{vbc,mapsimvol}. Useful algebraic tools in
the context of simplicial volume are so-called bounded
cohomology~\cite{vbc,ivanov} and $\ell^1$-homology~\cite{l1homology}.

\section{Strongly inflexible manifolds}\label{sec:stronglyinflexible}

We will now focus on the manifold aspects, discussing the basic terminology
of (strong) inflexibility and proving
Theorem~\ref{thm:stronglyinflexible}. Moreover, we will discuss the
meaning of this result in terms of secondary simplicial volume
(Section~\ref{subsec:secondarysimvol}).

\subsection{Terminology}

We first recall the definition of (strong) inflexibility~\cite{crowleyloeh}.

\begin{defi}[inflexibility]\label{def:sinflex}
  Let $M$ be an oriented closed connected manifold of dimension~$n \in \N$. 
  For oriented closed connected $n$-manifolds~$N$ we write 
  \[ D(N,M) := \bigl\{ \deg f \bigm| f \in \map(N,M)\bigr\}. 
  \]
  We call $M$ \emph{inflexible} if $D(M,M)$ is finite. We
  call $M$ \emph{strongly inflexible} if for every oriented closed
  connected $n$-manifold~$N$ the set~$D(N,M)$ is finite.
\end{defi}

For example, spheres and tori are flexible (i.e., not
inflexible). Oriented closed connected hyperbolic manifolds are
strongly inflexible~\cite{vbc,crowleyloeh}. In fact, all oriented closed
connected manifolds with non-zero simplicial volume are strongly
inflexible. More conceptually, all
finite functorial semi-norms provide obstructions to strong
inflexibility and vice versa~\cite{crowleyloeh,ffsnrep}. We will
discuss this relation in more detail in Section~\ref{subsec:genfsn}.

\subsection{Products of strongly inflexible manifolds}\label{subsec:prod}

In many examples it is known that products of inflexible manifolds are 
inflexible~\cite{crowleyloeh}. However, it is not clear that this holds 
in general. In the case of strongly inflexible manifolds, the situation 
simplifies as follows:

\begin{prop}[products of strongly inflexible manifolds]\label{prop:prod}
  Let $M_1$ and $M_2$ be oriented closed connected strongly inflexible
  manifolds. Then also $M_1 \times M_2$ is strongly
  inflexible.
\end{prop}

The proof is based on Thom's representation of homology classes by
manifolds and straightforward calculations in singular homology and
cohomology; similar arguments appear in related work on domination
of/by product manifolds~\cite{kotschickloeh,kln,neofytidis}.

\begin{proof}
  We abbreviate $n_1 := \dim M_1$, $n_2 := \dim M_2$ and $n := \dim M_1
  \times M_2 = n_1 + n_2$. Let $N$ be an oriented closed connected
  $n$-manifold. We need to show that $D(N, M_1 \times M_2)$ is
  finite. 

  As first step, we will prove that the sets
  \begin{align*}
    F_1 := 
    & \bigl\{ H_{n_1}(f_1;\Q) \bigm| f_1 \in \map(N,M_1) \bigr\}
      \subset \Hom_\Q \bigl( H_{n_1}(N;\Q), H_{n_1}(M_1;\Q)\bigr)
   \\
    F_2 :=
    & \bigl\{ H_{n_2}(f_2;\Q) \bigm| f_2 \in \map(N,M_2) \bigr\}
      \subset \Hom_\Q \bigl( H_{n_2}(N;\Q), H_{n_2}(M_2;\Q)\bigr)
  \end{align*}
  are finite. 

  Because $N$ is a closed manifold we know that $H_{n_1}(N;\Q)$ is
  finite dimensional. Let $B_1 \subset H_{n_1}(N;\Q)$ be a 
  $\Q$-basis. Let $\beta \in B_1$. In view of Thom's representation of
  homology classes by manifolds~\cite{thom} there exists an oriented
  closed connected $n_1$-manifold~$N_\beta$, a continuous map~$g_\beta
  \colon N_\beta \longrightarrow N$ and a~$k_\beta \in \Q\setminus \{0\}$ 
  with
  \[ H_{n_1}(g_\beta;\Q) [N_\beta]_\Q = k_\beta \cdot \beta \in H_{n_1}(N;\Q). 
  \]
  Because $M_1$ is strongly inflexible, the set $D(N_\beta, M_1)$ is 
  finite. Composition with $g_\beta$ shows therefore that also the 
  set
  \[ \bigl\{ d\in \Z \bigm| \exists_{f_1 \in \map(N,M_1)}\ H_{n_1}(f_1;\Q)(\beta) = d\cdot [M_1]_\Q 
     \bigr\} 
  \]
  is finite. Because $B_1$ is a finite $\Q$-basis of~$H_{n_1}(N;\Q)$ 
  we hence obtain that the set $F_1$ is finite. For the same reason 
  also $F_2$ is finite.

  As second step, we will now combine the finiteness of~$F_1$ and
  $F_2$ with the cohomological cross-product to derive finiteness
  of~$D(N,M_1 \times M_2)$.

  To this end let $f \in \map(N,M_1 \times M_2)$. We write
  \[ f_1 := p_1 \circ f \in \map(N,M_1) 
  \qquad\text{and}\qquad
     f_2 := p_2 \circ f \in \map(N,M_2),
  \]
  where $p_1 \colon M_1 \times M_2 \longrightarrow M_1$ and $p_2
  \colon M_1 \times M_2 \longrightarrow M_2$ are the projections onto
  the factors. We then obtain for the cohomological fundamental
  classes with $\Q$-coefficients that 
  \begin{align*}
    \deg f \cdot [M_1 \times M_2]^*_\Q 
    & = H^n(f;\Q)[M_1 \times M_2]^*_\Q \\
    & = \varepsilon \cdot H^n(f;\Q) \bigl( [M_1]^*_\Q \times [M_2]^*_\Q \bigr)\\
    & = \varepsilon \cdot H^n(f;\Q) \bigl( H^{n_1}(p_1;\Q) [M_1]^*_\Q \cup H^{n_2}(p_2;\Q)[M_2]^*_\Q \bigr)\\
    & = \varepsilon \cdot H^{n_1}(f_1;\Q)[M_1]^*_\Q \cup H^{n_2}(f_2;\Q)[M_2]^*_\Q, 
  \end{align*}
  where $\varepsilon := (-1)^{n_1 \cdot n_2}$.  In particular, the
  degree~$\deg f$ is uniquely determined by~$H^{n_1}(f_1;\Q)$ and
  $H^{n_2}(f_2;\Q)$, which in turn (by the universal coefficient
  theorem) are uniquely determined by~$H_{n_1}(f_1;\Q) \in F_1$ and
  $H_{n_2}(f_2;\Q) \in F_2$. Because the sets $F_1$ and $F_2$ are finite by the
  first step, we obtain that also $D(N,M_1\times M_2)$ is finite.
\end{proof}

\subsection{Strongly inflexible manifolds with trivial simplicial volume}\label{subsec:sinflexsimvol}

We will now give examples of stronlgy inflexible aspherical manifolds
whose simplicial volume is zero. We start with the case of
dimension~$3$ and then use inheritance of strong inflexibility under 
products to deal with the general case. 

\begin{exa}[dimension~$3$]\label{exa:sinflex3}
  Let $M$ be the total space of an orientable non-trivial $S^1$-bundle
  over an oriented closed connected surface of genus at
  least~$2$. Then $M$ is strongly inflexible (see
  Corollary~\ref{cor:3inflex} below) and aspherical. On the other
  hand, it is known that $\|M\| =0$ \cite{vbc,ivanov}.
\end{exa}

In combination with Proposition~\ref{prop:prod} we can now prove
Theorem~\ref{thm:stronglyinflexible}:

\begin{proof}[Proof of Theorem~\ref{thm:stronglyinflexible}]
  Let $M_1$ be an oriented closed connected aspherical $3$-manifold that is 
  strongly inflexible and satisfies~$\|M_1\| = 0$; such manifolds exist by 
  Example~\ref{exa:sinflex3}. 

  If $d \geq 5$, we take an oriented closed connected hyperbolic
  manifold~$M_2$ of dimension~$d-3$. Then $M_2$ is aspherical and
  $M_2$ is strongly inflexible. 

  In view of Proposition~\ref{prop:prod} also the product~$M_1 \times
  M_2$ is strongly inflexible. Moreover, by construction, $M_1 \times
  M_2$ is an oriented closed connected aspherical $d$-manifold and 
  we have~\cite{vbc} 
  \[ \| M_1 \times M_2\| \leq {d \choose 3} \cdot \|M_1\| \cdot \|M_2\| 
        = 0,
  \]
  as desired. 
\end{proof}

\begin{rem}
  Using fundamental properties of $\ell^1$-homology, one can see that
  the examples of strongly inflexible manifolds~$M$ constructed in the
  proof of Theorem~\ref{thm:stronglyinflexible} do not only have
  trivial simplicial volume but also are \mbox{\emph{$\ell^1$-invisible}},
  i.e., the image~$[M]^{\ell^1} \in H_*^{\ell^1}(M;\R)$ of~$[M]_\R$ in
  $\ell^1$-homology is the zero
  class~\cite[Example~6.7]{l1homology}. In other words,
  $\ell^1$-invisibility of manifolds does not imply weak flexiblity.
  Notice that it is an open problem whether all manifolds with trivial
  simplicial volume are $\ell^1$-invisible.
\end{rem}

We conclude this section with some open problems on strong
inflexibility. Generalising Example~\ref{exa:sinflex3}, one could ask:

\begin{question}
  Let $B$ be an oriented closed connected hyperbolic manifold and let
  $M \longrightarrow B$ be a non-trivial circle bundle over~$B$. Under
  which conditions will $M$ be strongly inflexible? 
\end{question}

In general, the fundamental class of~$B$ cannot be lifted to~$M$ and
hence does not provide a useful obstruction on~$M$. However, one could
try to use a lift in $\ell^1$-homology. More concretely: Let $p \colon
M \longrightarrow B$ be a circle bundle. Then the induced
map~$H_*^{\ell^1}(p;\R) \colon H^{\ell^1}_*(M;\R) \longrightarrow
H^{\ell^1}_*(B;\R)$ is an isometric isomorphism~\cite{l1homology}. In
particular, if $M$ and $B$ are oriented closed connected manifolds of
dimension~$n$ and~$n-1$ respectively, we obtain the codimension~$1$
class
\[ \alpha := H^{\ell^1}_{n-1}(p;\R)^{-1}([B]^{\ell^1}) \in H^{\ell^1}_{n-1}(M;\R)
\]
in $\ell^1$-homology of~$M$. For example, in the case that $B$ is a
hyperbolic surface and $p$ is a non-trivial bundle, this class was
considered by Derbez in the study of local rigidity of aspherical
$3$-manifolds~\cite{derbezlocal}. More generally, if $B$ is an oriented 
closed connected hyperbolic manifold, the class $\alpha$ will be 
non-trivial -- it will even have non-zero $\ell^1$-semi-norm. Can 
this class be used as an obstruction to prove strong inflexibility of~$M$\;?

At the other extreme, it also remains an open problem to determine whether 
there exist simply connected strongly inflexible manifolds (of non-zero dimension).

\subsection{Secondary simplicial volume}\label{subsec:secondarysimvol}

Secondary simplicial volume is a refinement of simplicial volume that 
allows to give refined information about vanishing of simplicial volume.

\begin{defi}[secondary simplicial volume]
  Let~$M$ be an oriented closed connected manifold of dimension~$n\in\N$. 
  The \emph{secondary simplicial volume of~$M$} is defined to
  be the integral sequence
  \begin{align*}
    \Sigma(M):=\bigl(\bigl\|k\cdot
    [M]_\Z\bigr\|_{1,\Z}\bigr)_{k\in\N},
  \end{align*}
  where $\|\cdot\|_{1,\Z}$ is the semi-norm on~$H_n(\args;\Z)$ induced
  by the $\Z$-valued \mbox{$\ell^1$-norm} on~$C_n(\args;\Z)$.
\end{defi}

\begin{rem}
  For all oriented closed connected manifolds~$M$ 
  the following holds~\cite[Remark~5.4]{ffsnrep}:
  \begin{align*}
    ||M||=\inf_{k\in\N_{>0}}{\frac{1}{k}\cdot
    \bigl\|k\cdot [M]_\Z\bigr\|_{1,\Z}}.
  \end{align*}
  In particular, vanishing of the simplicial volume of~$M$ can be 
  expressed in terms of the growth behaviour of the sequence~$\Sigma(M)$.
\end{rem}

 The strongest vanishing of simplicial volume occurs if the secondary
 simplicial volume contains a bounded subsequence. 
 We recall the complete geometric characterisation of such manifolds in
 Propositions~\ref{prop:bdsecsimvol} and~\ref{prop:trivsecsimvol} 
 in terms of flexibility.

\begin{prop}[{\cite[Corollary~5.5]{ffsnrep}}]\label{prop:bdsecsimvol}
  Let~$M$ be an oriented closed connected manifold of dimension~$n\in\N$.
  Then the following are equivalent:
  \begin{enumerate}
    \item The manifold~$M$ is weakly flexible.
    \item The secondary simplicial volume~$\Sigma(M)$ 
      contains a
      bounded subsequence.
    \item All finite functorial semi-norms on~$H_n(\args;\R)$ vanish
      on~$[M]_\R$.
  \end{enumerate}
\end{prop}

We will now focus on the $3$-dimensional case. 
The classification of weakly flexible 3-manifolds by Derbez, Sun, 
and Wang~\cite{dsw} translates into the following result:

\begin{cor}\label{cor:3inflex}
  Let~$M$ be an oriented closed connected $3$-manifold. Then the following
  are equivalent:
  \begin{enumerate}
    \item The secondary simplicial volume~$\Sigma(M)$ 
      contains a
      bounded subsequence.
    \item Each prime summand of~$M$ is covered by a torus bundle
      over~$S^1$, by a trivial $S^1$-bundle or by~$S^3$.
  \end{enumerate}
  In particular, the $3$-manifolds from Example~\ref{exa:sinflex3} are
  strongly inflexible.
\end{cor}

\begin{rem}
  In the previous section we observed that the $3$-manifolds from
  Example~\ref{exa:sinflex3} even give examples for $\ell^1$-invisible
  manifolds with non-trivial secondary simplicial volume. In contrast,
  we asked whether all manifolds with trivial secondary simplicial volume
  are $\ell^1$-invisible. In dimension~$3$, this easily follows from the
  characterization above and the following two facts~\cite[Example~6.7]{l1homology}:
  \begin{itemize}
    \item Total spaces of fibrations of oriented closed connected manifolds
    with fibre an oriented closed connected manifold that has amenable
    fundamental group are $\ell^1$-invisible.
    \item The connected sum of two $\ell^1$-invisible $3$-manifolds is
    $\ell^1$-invisible.
  \end{itemize}
\end{rem}

\begin{prop}[{\cite[Theorem~3.2]{smallisv}}]\label{prop:trivsecsimvol}
  Let~$M$ be an oriented closed connected manifold of dimension $n\in\N_{>0}$.
  Then the following are equivalent:
  \begin{enumerate}
    \item The secondary simplicial volume~$\Sigma(M)$ 
      contains a
      bounded subsequence with bound~$1$.
    \item The manifold~$M$ is dominated by~$S^n$ and~$n$ is odd.
  \end{enumerate}
\end{prop}

\begin{cor}
  Let~$M$ be an oriented closed connected 3-manifold. Then the following
  are equivalent:
  \begin{enumerate}
    \item The secondary simplicial volume~$\Sigma(M)$ 
      contains a
      bounded subsequence with bound~$1$.
    \item The manifold~$M$ is spherical, i.e., finitely covered by~$S^3$.
  \end{enumerate}
\end{cor}
\begin{proof}
  If $f \colon N \longrightarrow M$ is a map of non-trivial degree
  between oriented closed connected manifolds of the same dimension,
  then the image of~$\pi_1(f)$ is a finite index subgroup
  in~$\pi_1(M)$.  Therefore, it follows by the Elliptization Theorem,
  that domination by~$S^3$ and being finitely covered by~$S^3$ is
  equivalent for oriented closed connected $3$-manifolds.
\end{proof}

\section{Exotic finite functorial semi-norms}\label{sec:exotic}

We will first recall how strongly inflexible manifolds generate
interesting functorial semi-norms (Section~\ref{subsec:genfsn}).
Combining this construction with Theorem~\ref{thm:stronglyinflexible}
will then complete the proof of Theorem~\ref{thm:exotic}
(Section~\ref{subsec:exotic}).

\subsection{Generating functorial semi-norms via manifolds}\label{subsec:genfsn}

We will apply the following principle to generate exotic functorial
semi-norms:

\begin{prop}[domination semi-norm associated with a manifold~\protect{\cite[Section~7.1]{crowleyloeh}}]
  Let $d \in \N$ and let $M$ be an oriented closed connected
  $d$-manifold. Then there exists a functorial semi-norm~$|\cdot|_M$ 
  on~$H_d(\args;\R)$ satisfying
  \[ \bigl|[N]_\R\bigr|_M = \sup \bigl\{ |D| \bigm| D \in D(N,M)\bigr\} \in \R_{\geq 0} \cup \{\infty\}
  \]
  for all oriented closed connected $d$-manifolds~$N$. The functorial
  semi-norm~$|\cdot|_M$ is finite if and only if $M$ is strongly
  inflexible.
\end{prop}

For instance, in this way, functorial semi-norms on~$H_{64}(\args;\R)$
have been constructed that are non-trivial on certain classes of
simply connected spaces~\cite[Theorem~1.2]{crowleyloeh}. However, it
is not known whether these are examples of finite functorial
semi-norms.

\subsection{Construction of exotic finite functorial semi-norms}\label{subsec:exotic}

We finally complete the proof of Theorem~\ref{thm:exotic}. 

\begin{proof}[Proof of Theorem~\ref{thm:exotic}]
  Let $d \in \{3\} \cup \N_{\geq 5}$. By Theorem~\ref{thm:stronglyinflexible} there 
  exists a strongly inflexible oriented closed connected $d$-manifold~$M$ with~$\| M \| = 0$. 
  Let $|\cdot|_M$ be the associated domination semi-norm on~$H_d(\args;\R)$; 
  because $M$ is strongly inflexible, this functorial semi-norm is indeed finite. 
  By construction, we have
  \[ \bigl| [M]_\R \bigr|_M = \sup \bigl\{ |D| \bigm| D \in D(M,M)\bigr\}  = 1, 
  \]
  but 
  $ \bigl\| [M]_\R \bigr\|_1 = \|M\| = 0.
  $
  In particular, $|\cdot|_M$ is \emph{not} carried by the
  $\ell^1$-semi-norm in the sense of Definition~\ref{def:carrier} 
  (this holds even on fundamental classes of aspherical manifolds).
\end{proof}

\begin{question}
  Does there exist a finite functorial semi-norm on~$H_d(\args;\R)$ that 
  carries all other finite functorial semi-norms on~$H_d(\args;\R)$\;?
\end{question}


\medskip
\vfill

\noindent
\emph{Clara L\"oh}\\
\emph{Daniel Fauser}\\[.5em]
  {\small
  \begin{tabular}{@{\qquad}l}
    Fakult\"at f\"ur Mathematik, 
    Universit\"at Regensburg, 
    93040 Regensburg\\
    \textsf{clara.loeh@mathematik.uni-r.de}, 
    \textsf{http://www.mathematik.uni-r.de/loeh}\\
    \textsf{daniel.fauser@mathematik.uni-regensburg.de} 
  \end{tabular}}
\end{document}